\documentclass[12pt]{amsart}
\usepackage{SSdefn}
\usepackage[lite,nobysame,alphabetic]{amsrefs}
\usepackage{verbatim}

\newcommand{\Zar}{\mathrm{Zar}}

\renewcommand{\ff}{\mathbf f}

\newcommand{\FI}{\mathbf{FI}}

\DeclareMathOperator{\gin}{gin}

\DeclareMathOperator{\HN}{\rH\rN}

\newtheorem*{theorem*}{Theorem}

\title[Generalizations of Stillman's conjecture via twisted commutative algebras]{Generalizations of Stillman's conjecture via \\ twisted commutative algebra}

\author{Daniel Erman}
\address{Department of Mathematics, University of Wisconsin, Madison, WI}
\email{\href{mailto:derman@math.wisc.edu}{derman@math.wisc.edu}}
\urladdr{\url{http://math.wisc.edu/~derman/}}

\author{Steven V Sam}
\address{Department of Mathematics, University of Wisconsin, Madison, WI}
\email{\href{mailto:svs@math.wisc.edu}{svs@math.wisc.edu}}
\urladdr{\url{http://math.wisc.edu/~svs/}}

\author{Andrew Snowden}
\address{Department of Mathematics, University of Michigan, Ann Arbor, MI}
\email{\href{mailto:asnowden@umich.edu}{asnowden@umich.edu}}
\urladdr{\url{http://www-personal.umich.edu/~asnowden/}}

\thanks{DE was partially supported by NSF DMS-1302057 and NSF DMS-1601619. SS was partially supported by NSF DMS-1500069. AS was supported by NSF DMS-1303082 and DMS-1453893 and a Sloan Fellowship.}

\date{April 24, 2018}


\begin{document}

\begin{abstract}
Combining recent results on noetherianity of twisted commutative algebras by Draisma and the resolution of Stillman's conjecture by Ananyan--Hochster, we prove a 
broad generalization of Stillman's conjecture. Our theorem yields an array of boundedness results in commutative algebra that only depend on the degrees of the generators of an ideal, and not the number of variables in the ambient polynomial ring.
\end{abstract}

\maketitle

\section{Introduction}

The introduction of noetherianity conditions in commutative algebra streamlined and generalized boundedness results in invariant theory. We revisit this theme: using Draisma's recent noetherianity result for twisted commutative algebras, and the resolution of Stillman's conjecture by Ananyan--Hochster, we prove a boundedness result for a large class of ideal invariants. This can be seen as a far reaching generalization of Stillman's conjecture.

As preparation for this result, we also establish a number of technical statements about certain topological spaces with group actions relevant to the study of ideal invariants. 

\subsection{Ideal invariants}

Fix a field $\bk$. An {\bf ideal invariant} (over $\bk$) is a rule $\nu$ that associates a quantity $\nu(I) \in \bZ \cup \{\infty\}$ to every homogeneous ideal $I$ in every standard-graded polynomial ring $A=\bk[x_1, \ldots, x_n]$, such that $\nu(I)$ only depends on the pair $(A, I)$ up to isomorphism. There are countless examples of ideal invariants: degree, projective dimension, regularity, the $(i,j)$ Betti number, etc. 

Let $\bd=(d_1, \ldots, d_r)$ be a tuple of positive integers. We say that an ideal $I$ is {\bf type $\bd$} if it is generated by $f_1, \ldots, f_r$ where $f_i$ is homogeneous of degree $d_i$. We say that an ideal invariant $\nu$ is {\bf bounded in degree $\bd$} if there exists $B \in \bZ$ such that for every type $\bd$ ideal $I \subset A$ we have $\nu(I) \le B$ or $\nu(I)=\infty$. We say that $\nu$ is {\bf degreewise bounded} if it is bounded in degree $\bd$ for all $\bd$. The main point of this definition is that the bound is independent of the number of variables.

There are two ``niceness'' conditions we require on our ideal invariants. We say that $\nu$ is {\bf cone-stable} if $\nu(I[x])=\nu(I)$ for all $(A,I)$, that is, adjoining a new variable does not affect the invariant. We say that $\nu$ is {\bf weakly upper semi-continuous} if the following holds: given a polynomial ring $A$, a variety $S$ over $\bk$, and a homogeneous ideal sheaf $\cI$ of $\cA=\cO_S \otimes_{\bk} A$ such that $\cA/\cI$ is $\cO_S$-flat, the map $s \mapsto \nu(\cI_s)$ is upper semi-continuous as $s$ varies over the geometric points of $S$; that is, for each $n$, the locus of geometric points $\{ s \mid \nu(\cI_s) \ge n\}$ is Zariski-closed. (Here $\cI_s$ denotes the fiber of $\cI$ at $s$, which is an ideal of $A$ by the flatness assumption.)

Stillman's conjecture is exactly the statement that the ideal invariant ``projective dimension'' (which is cone-stable and weakly upper semi-continuous) is degreewise bounded. We prove the following generalization:

\begin{theorem}\label{thm:invariants}
Any ideal invariant that is cone-stable and weakly upper semi-continuous is degreewise bounded.
\end{theorem}

\subsection{The space of ideals}

We introduce a topological space $Y_{\bd}$ that parametrizes isomorphism classes of type $\bd$ ideals in the infinite polynomial ring $\bk[x_1,x_2,\ldots]$. We construct $Y_{\bd}$ as a quotient of an infinite dimensional variety $X_{\bd}$ that parametrizes generating sets $(f_1,\dots,f_r)$ of type $\bd$ ideals. Theorem~\ref{thm:invariants} is deduced from two results about $Y_{\bd}$.

\begin{theorem} \label{thm:noeth}
The space $Y_{\bd}$ is noetherian.
\end{theorem}

\begin{theorem} \label{thm:strat}
The space $Y_{\bd}$ admits a finite stratification $\{Y^{\lambda}_{\bd}\}_{\lambda \in \Lambda}$ such that the universal quotient ring is flat over each stratum $($see \S \ref{ss:strat} for the precise meaning of this$)$.
\end{theorem}

The primary input into the proof of Theorem~\ref{thm:noeth} is Draisma's theorem \cite{draisma} on noetherianity of polynomial representations, while the primary input into the proof of Theorem~\ref{thm:strat} is the resolution of Stillman's conjecture by Ananyan--Hochster \cite{ananyan-hochster}.

We now sketch the proof of Theorem~\ref{thm:invariants}. Suppose $\nu$ is a cone-stable weakly upper semi-continuous ideal invariant. By cone-stability, $\nu$ defines a function on $Y_{\bd}$. Let $Z_n \subset Y_{\bd}$ be the locus where $\nu \ge n$. These loci form a descending chain. By weak upper semi-continuity and Theorem~\ref{thm:strat}, $Z_n \cap Y^{\lambda}_{\bd}$ is closed in $Y^{\lambda}$. By Theorem~\ref{thm:noeth}, $Y^{\lambda}_{\bd}$ is noetherian, and so $Z_{\bullet} \cap Y^{\lambda}_{\bd}$ stabilizes for each $\lambda$. Thus $Z_{\bullet}$ stabilizes, which shows that $\nu$ is bounded.

\begin{remark}
As should be clear, our proof of Theorem~\ref{thm:invariants} is a consequence of Stillman's conjecture.  A direct proof of Stillman's conjecture using~\cite{draisma} appears in~\cite{ess-stillman}.  See also Remark~\ref{remark:strat compare} for a comparison of the stratification in Theorem~\ref{thm:strat} and the one appearing in~\cite[Theorem~5.13]{ess-stillman}.
\end{remark}

Surprisingly, one can go even further. In \S\ref{ss:improve}, we define a space $Y_{\le d}$ which parametrizes isomorphism classes of {\it all} finitely generated homogeneous ideals of $\bk[x_1,x_2,\dots]$ generated in degrees $\le d$, with no restriction on the number of generators.

\begin{theorem} \label{thm:noeth 2}
The space $Y_{\le d}$ is noetherian.
\end{theorem}

Unfortunately, the analog of Theorem~\ref{thm:strat} fails for $Y_{\le d}$ (it fails already for $d=1$). Nonetheless, one can deduce a certain analog of Theorem~\ref{thm:invariants} from Theorem~\ref{thm:noeth 2} (see Proposition~\ref{prop:invar2}).

\subsection{Topological comparison results}

To define the space $Y_\bd$, we realize it as the quotient space of another space $X_\bd$, which parametrizes type $\bd$ ideals together with a choice of generating set of type $\bd$. The space $X_\bd$ can also be described as the direct sum of the symmetric powers $\Sym^{d_i}(\bk^\infty)$. Each summand $\Sym^{d_i}(\bk^\infty)$ is the space of homogeneous degree $d_i$ polynomials in infinitely many variables. A key property that we use is that these spaces (and finite direct sums of them) are noetherian under the change of basis action given by $\GL = \bigcup_n \GL_n$. We deduce this from a recent result of Draisma \cite{draisma}. However, Draisma's results is not directly applicable: the space we describe here is the direct limit of the space of homogeneous polynomials in $n$ variables with $n \to \infty$, but Draisma's result applies to the corresponding inverse limit spaces.

Hence in \S\ref{s:polyvar}, we prove some topological comparison results that show that noetherianity of direct limit spaces and noetherianity of inverse limit spaces are equivalent to one another under suitable conditions which include the situation above. Namely, we work in the more general context of finite length polynomial functors and consider the corresponding spaces together with the change of basis action of $\GL$. There are a number of natural topologies that one might consider, and the key principle is that it does not matter which one is chosen as long as one works equivariantly.

\subsection{Outline}

In \S \ref{s:polyvar}, we collect the results we need about certain infinite dimensional varieties. In \S \ref{s:thms}, we prove the main theorems of the paper. In \S \ref{s:examples}, we give examples of degreewise bounded invariants. Finally, \S \ref{s:comments} has some further comments and generalizations.

\subsection*{Acknowledgments}

We thank Brian Lehmann and Claudiu Raicu for conversations inspiring \S\ref{ss:fano}, Giulio Caviglia for observations that simplified the proof of Theorem~\ref{thm:strat}, and Bhargav Bhatt and Mel Hochster for additional useful conversations. We thank BIRS for hosting us during the workshop ``Free Resolutions, Representations, and Asymptotic Algebra'', as some of the ideas in this paper were refined during that meeting.

\section{Varieties defined by polynomial functors} \label{s:polyvar}

\subsection{Setup}
For the purposes of this paper, a {\bf polynomial functor} over a field $\bk$ is an endofunctor of the category of vector spaces that is a subquotient of a direct sum of tensor power functors. When $\bk$ has characteristic~0, every polynomial functor is a direct sum of Schur functors. The category of polynomial functors is abelian, and comes equipped with a tensor product defined by
\[
  (F \otimes G)(U)=F(U) \otimes G(U).
\]
This tensor product is equipped with the symmetry that interchanges $F(U)$ and $G(U)$.

Fix a finitely generated algebra object $\ul{R}$ in the category of polynomial functors. Examples of such algebra objects are easy to come by: if $\ul{V}$ is a finite length polynomial functor then $\Sym(\ul{V})$ is such an algebra object. In fact, these are the only examples relevant to this paper. Let $R_n=\ul{R}(\bk^n)$, a finitely generated $\bk$-algebra, and let $X_n=\Spec(R_n)$.

The standard inclusion $\bk^n \to \bk^{n+1}$, given by $(x_1,\dots,x_n) \mapsto (x_1,\dots,x_n,0)$, induces an inclusion $R_n \to R_{n+1}$ and thus a projection $X_{n+1} \to X_n$. We let $R$ be the direct limit of the $R_n$, and let $\wh{X}$ be the inverse limit of the $X_n$'s in the category of schemes, which is simply the affine scheme $\Spec(R)$. We let $\vert \wh{X} \vert$ be the topological space underlying the scheme $\wh{X}$, which is the inverse limit of the spaces $\vert X_n \vert$.

The standard projection $\bk^{n+1} \to \bk^n$, given by $(x_1,\dots,x_n,x_{n+1}) \mapsto (x_1,\dots,x_n)$, induces a surjection $R_{n+1} \to R_n$ and thus a closed immersion $X_n \to X_{n+1}$. We let $\wh{R}$ be the inverse limit of the $R_n$'s, regarded as a topological ring, and we let $X$ be the ind-scheme defined by the directed system $\{X_n\}$. Let $\vert X \vert$ be the direct limit of the sets $\vert X_n \vert$. We define the {\bf ind-topology} on $\vert X \vert$ to be the direct limit topology, and let $\vert X \vert^{\ind}$ denote the resulting topological space. The set $\vert X \vert$ is canonically identified with the set of open prime ideals in the topological ring $\wh{R}$. In this way, $\vert X \vert$ is a subset of $\Spec(\wh{R})$, and one can give it the subspace topology. We call this the {\bf Zariski topology}, and denote the resulting space by $\vert X \vert^{\Zar}$. The Zariski topology is, in many situations, the ``correct'' topology, since its closed sets are zero loci of equations, but it is often easier to check that a set is closed in the ind-topology, since this can be checked in each $X_n$ separately. Every Zariski closed set is ind-closed, but the converse is not true in general; see \cite{anderson} for an example.

By definition, $R_n$ maps to $R$. However, there is also a canonical map in the opposite direction. Indeed, for each $m \ge n$ the standard projection $\bk^m \to \bk^n$ induces a ring homomorphism $R_m \to R_n$. These are compatible, and thus define a map $R \to R_n$. The composition $R_n \to R \to R_n$ is the identity, and so the map $R_n \to R$ is injective while the map $R \to R_n$ is surjective. These surjections are compatible as $n$ varies, and thus define a map $R \to \wh{R}$. By the same reasoning, we get a map $X \to \wh{X}$.

Recall that a topological space $Y$ equipped with an action of a group $G$ is {\bf topologically $G$-noetherian} if every descending chain of $G$-stable closed subsets stabilizes. Draisma proved that, over any field $\bk$, $\vert \wh{X} \vert$ is topologically $\GL$-noetherian \cite[Theorem 1]{draisma}, where here, and in what follows, $\GL$ denotes the algebraic group $\bigcup_{n \ge 1} \GL_n$.  The goal of this section is to transfer Draisma's result to the spaces $\vert X \vert^{\Zar}$ and $\vert X \vert^{\ind}$ and some related spaces.

\subsection{Comparison of $X$ and $\wh{X}$}

We begin by comparing $\wh{X}$ to the Zariski topology on $X$.

\begin{proposition} \label{prop:top1}
We have a bijection
\begin{displaymath}
\{ \text{$\GL$-stable closed subsets of $\vert \wh{X} \vert$} \} \to \{ \text{$\GL$-stable closed subsets of $\vert X \vert^{\Zar}$} \}
\end{displaymath}
given by $Z \mapsto Z \cap \vert X \vert$. In particular, $\vert X \vert^{\Zar}$ is $\GL$-noetherian.
\end{proposition}

We refer to elements of the image of the map $R \to \wh{R}$ as {\bf finite polynomials}.

\begin{lemma} \label{lem:lem1}
Suppose $I \subset \wh{R}$ is a closed, $\GL$-stable ideal. Then every element of $I$ is the limit of finite polynomials that also belong to $I$.
\end{lemma}

\begin{proof}
Any continuous endomorphism of $\widehat{\bk}^{\infty} = \varprojlim \bk^n$ acts on $\widehat{R}$, and any $\GL$-stable subspace of $\widehat{R}$ (such as $I$) is automatically stable by these additional operators. Let $f \in I$, and let $f_n$ be its image under $\widehat{R} \to R_n \to \widehat{R}$. The composite map $\widehat{R} \to \widehat{R}$ is the action of a continuous endomorphism $\widehat{\bk}^{\infty} \to \bk^n \to \widehat{\bk}^{\infty}$, and so it maps $I$ into itself. We thus see that $f_n \in I$. Since $f$ is the limit of the $f_n$'s, the result follows.
\end{proof}

\begin{lemma} \label{lem:lem2}
If $Z \subset \wh{X}$ is a $\GL$-stable closed subset, then every point of $Z$ is a limit of points of $Z \cap X$.
\end{lemma}

\begin{proof} 
Let $I \subset R$ be the ideal defining $Z$. As in the previous proof, any endomorphism of $\bk^{\infty}$ acts on $R$ and carries $I$ to itself. Let $f_n \colon \wh{X} \to \wh{X}$ be the map induced by the projection $\bk^{\infty} \to \bk^n \to \bk^{\infty}$. We thus see $f_n(Z) \subset Z$. Let $x \in Z$ and put $x_n=f_n(x)$. The map $f_n$ factors as $\wh{X} \to X_n \to X \to \wh{X}$, and so $x_n \in Z \cap X$ for all $n$. Since $x$ is the limit of the $x_n$, the result follows.
\end{proof}

\begin{proof}[Proof of Proposition~\ref{prop:top1}]
Define a map in the opposite direction by sending a $\GL$-stable Zariski closed set $Z$ of $\vert X \vert$ to its closure in $\vert \wh{X} \vert$. We claim that the two maps are inverse.

Start with $Z\subseteq \vert X \vert$. Let $Z'$ be the closure of $Z$ in $\vert \wh{X} \vert$ and let $Z''=Z' \cap \vert X \vert$.  Clearly $Z \subset Z''$. Let $x$ be a point of $Z''$. We must show that every infinite polynomial that vanishes on $Z$ also vanishes on $x$. If $f$ is some infinite polynomial vanishing on $Z$, then by Lemma~\ref{lem:lem1}, it is a limit of finite polynomials $f_n$ vanishing on $Z$. Each $f_n$ vanishes on $x$, and hence so does the limiting polynomial $f$. So $x \in Z$.

Now consider $W\subseteq \vert \wh{X} \vert$. Let $W'$ be the closure of $W \cap \vert X \vert$ in $\vert \wh{X} \vert$. Then clearly $W' \subset W$. Let $x$ be a point of $W$. Lemma~\ref{lem:lem2} shows that every point in $W$ is a limit of points from $W\cap Y$ and hence we obtain the opposite containment as well.

We have thus shown that $Z \mapsto Z \cap \vert X \vert$ defines a bijection as in the statement of the proposition. Draisma's theorem shows that the source of this bijection satisfies the descending chain condition. Thus the target does as well, which shows that $\vert X \vert^{\Zar}$ is $\GL$-noetherian.
\end{proof}

\subsection{Comparison of the ind- and Zariski topologies}

The Zariski and ind-topologies on an ind-scheme are typically very different. We show that this distinction disappears in our situation, where we focus on $\GL$-stable subsets.

\begin{proposition} \label{prop:top2}
A $\GL$-stable subset of $\vert X \vert$ is Zariski closed if and only if it is ind-closed. In particular, $\vert X \vert^{\ind}$ is $\GL$-noetherian.
\end{proposition}

\begin{lemma} \label{lem:radical}
Let $I$ be a $\GL$-stable radical ideal of $R$. Then its image in $R_n$ is radical, for any $n$.
\end{lemma}

\begin{proof}
Let $I_n$ be the image of $I$ in $R_n$. We claim that $I_n = I \cap R_n$. Since the composite $R_n \to R \to R_n$ is the identity, it is clear that any element of $I \cap R_n$ is contained in $I_n$. Conversely, the projection map $R_m \to R_n \subset R_m$ is obtained by applying the projection map $\bk^m \to \bk^n \subset \bk^m$, which can be realized as the limit of elements of $\GL_m$, and so maps $I \cap R_m$ into itself. We thus see that any element of $I_n$ belongs to $I$, which shows $I_n \subset I \cap R_n$.

Now suppose $x \in R_n$ and $x^k \in I_n$. Since $I$ is a radical ideal and $x^k \in I$, we have $x \in I$. Thus $x \in I \cap R_n = I_n$, and so $I_n$ is a radical ideal of $R_n$.
\end{proof}

\begin{proof}[Proof of Proposition~\ref{prop:top2}]
Every Zariski closed set is ind-closed (even without $\GL$-stability), so it suffices to show that if $Z \subset \vert X \vert$ is $\GL$-stable and ind-closed then it is Zariski closed. Let $Z_n= Z \cap X_n$, a Zariski closed subset of $X_n$. Let $I_n \subset R_n$ be the unique radical ideal such that $V(I_n)=Z_n$. Then $I_n$ is $\GL_n$-stable. Let $I^{(n)}$ be the smallest $\GL$-stable radical ideal of $R$ that contains $I_n$. Its zero locus is the intersection of the $\GL$-translates of the inverse image of $Z_n$ in $X$. Hence, if $m \ge n$, the image of $I^{(m)}$ in $R_n$ defines $Z_n$, since $Z_m \cap X_n = Z_n$. But this image is radical, by Lemma~\ref{lem:radical}, so $I^{(m)} \cap R_n = I_n$. If $m<n$, the image of $I^{(m)}$ in $R_n$ is a closed subset of $Z_n$. Again, $I^{(m)} \cap R_n$ is radical by Lemma~\ref{lem:radical}, so $I^{(m)} \cap R_n \supseteq I_n$. Now let $J = \bigcap_{n \ge 1} I^{(n)}$. Then $J \cap R_n = I_n$, so $V(J) = Z$ and we see that $Z$ is Zariski-closed.
\end{proof}

\subsection{Some variants}

Let $R_n^*=\ul{R}((\bk^n)^*)$, where $(\bk^n)^*$ denotes the dual space to $\bk^n$. Of course, $R_n^*$ is isomorphic to $R_n$ as a $\bk$-algebra, but the action of $\GL_n$ is different. Let $X^*$ and $\wh{X}^*$ be defined analogously to before. Given an action of $\GL$ or $\GL_n$ on some object, we define the {\bf conjugate action} as the precomposition with the automorphism $g \mapsto {}^tg^{-1}$.

\begin{proposition}
We have isomorphisms of $($ind-$)$schemes $\wh{X} \to \wh{X}^*$ and $X \to X^*$ that are $\GL$-equivariant for the conjugate action on the source and the standard action on the target.
\end{proposition}

\begin{proof}
Let $e_1, \ldots, e_n$ be the standard basis for $\bk^n$ and $e_1^*, \ldots, e_n^*$ the dual basis for $(\bk^n)^*$. We have a linear isomorphism $i_n \colon \bk^n \to (\bk^n)^*$ taking $e_i$ to $e_i^*$. This is $\GL_n$-equivariant using the conjugate action on the source and the standard action on the target. The $i_n$ thus induce the requisite isomorphisms $\wh{X} \to \wh{X}^*$ and $X \to X^*$.
\end{proof}

A subset of $\vert \wh{X} \vert$ or $\vert X \vert$ is stable under the standard $\GL$-action if and only if it is stable under the conjugate $\GL$-action. Thus, by Proposition~\ref{prop:top2} and its corollary, we find:

\begin{corollary}\label{cor:X*ind}
A $\GL$-stable subset of $\vert X^* \vert$ is ind-closed if and only if it is Zariski closed.
\end{corollary}

\begin{corollary} \label{cor:conjtop}
The spaces $\vert \wh{X}^* \vert$, $\vert X^* \vert^{\Zar}$, and $\vert X^* \vert^{\ind}$ are topologically $\GL$-noetherian.
\end{corollary}

\begin{example} \label{ex:poly}
Let $\ul{V}$ be a finite length polynomial functor and let $\ul{R}=\Sym(\ul{V})$. Let $V_n=\ul{V}(\bk^n)$ and let $V=\varinjlim V_n$. Also let $V_*=\varinjlim V_n^*$ be the restricted dual of $V$, where the limit is taken with respect to the standard inclusions of $(\bk^n)^*$ into $(\bk^{n+1})^*$. Then we have canonical identifications
\begin{displaymath}
\wh{X}=V^*, \qquad X=V_*, \qquad \wh{X}^*=(V_*)^*, \qquad X^*=V,
\end{displaymath}
where here $V^*$ and $(V_*)^*$ are the usual linear duals of the spaces $V$ and $V_*$.
\end{example}

\begin{remark}
The moral of this section is that all limit topological spaces one can sensibly form from $\ul{R}$ are essentially equivalent when working equivariantly. This heuristic does \emph{not} hold in some similar situations; see \cite[\S 4]{eggsnow} for examples.
\end{remark}

\section{Theorems about $Y_{\bd}$} \label{s:thms}

\subsection{Notation}

Let $A_n$ be the polynomial ring $\bk[x_1, \ldots, x_n]$ and let $A$ be the infinite polynomial ring $\bk[x_1, x_2, \ldots]$. For a tuple $\ff=(f_1,\ldots, f_r)$ of elements of $A_n$, we let $I_{\ff,n}$ be the ideal of $A_n$ that they generate, and let $I_{\ff}$ be the ideal of $A$ that they generate. Let $B_{\ff,n}=A_n/I_{\ff,n}$ and $B_{\ff}=A/I_{\ff}$.

For an integer $d>0$, let $X_{d,n}=\Sym^d(\bk^n)$, regarded as an affine scheme. For a degree tuple $\bd=(d_1,\dots,d_r)$, we let
\begin{displaymath}
X_{\bd,n}=X_{d_1,n} \times \cdots \times X_{d_r,n}
\end{displaymath}
and we let $X_{\bd}$ be the ind-scheme defined by the system $(X_{\bd,n})_n$. This fits into the variant setup of the previous section, since $X_{\bd}$ is the scheme $X^*$ from Example~\ref{ex:poly} with $\ul{V}=\Sym^{d_1} \oplus \cdots \oplus \Sym^{d_r}$. Let $\cA_{\bd,n}$ be the sheaf of algebras $A_n \otimes \cO_{X_{\bd,n}}$ on $X_{\bd,n}$. A $\bk$-point of $X_{\bd,n}$ corresponds to a tuple $\ff=(f_1, \ldots, f_r)$ of elements of $A_n$. The family of ideals $I_{\ff,n}$ assembles to an ideal sheaf $\cI_{\bd,n}$ of $\cA_{\bd,n}$ (meaning that the image of the fiber $\cI_{\bd,n}(\ff)$ in the fiber $\cA_{\bd,n}(\ff)=A_n$ is $I_{\ff,n}$). We let $\cB_{\bd,n}$ be the quotient sheaf; its fiber at $\ff$ is $B_{\ff,n}$.

\subsection{The space $Y_{\bd}$}

Let $Y_{\bd}$ be the set of isomorphism classes of type $\bd$ ideals in $A$, where we say that ideals $I$ and $J$ are isomorphic if there exists an isomorphism $\sigma \colon A \to A$ of graded rings with $\sigma(I)=J$. Let $X_{\bd}^{\circ}$ be the set of closed points in $X_{\bd}$ and let $\pi \colon X_\bd^{\circ} \to Y_\bd$ be the map taking a tuple $\ff$ to the class of the ideal $I_{\ff}$ that it generates.  The map $\pi$ is surjective and $\GL$-invariant. We give $Y_{\bd}$ the induced topology, using the Zariski topology on $X_{\bd}^{\circ}$. Thus a subset $Z$ of $Y_{\bd}$ is closed if and only if $\pi^{-1}(Z)$ is Zariski closed in $X_{\bd}^{\circ}$. Since $\pi^{-1}(Z)$ is $\GL$-stable, it is Zariski closed if and only if it is ind-closed (Proposition~\ref{prop:top2}), so the ind-topology on $X_{\bd}^{\circ}$ induces the same topology on $Y_{\bd}$.

\begin{remark}
If $\bd=(d,\ldots,d)\in \bZ^r$, then two tuples $\ff,\bg \in X_{\bd}$ generate the same ideal in $A$ if and only if they differ by an element of $\GL_r$, and generate isomorphic ideals if and only if they differ by an element of $\GL \times \GL_r$. We thus see that $Y_{\bd}$ is the quotient of $X_{\bd}^\circ$ by the group $\GL \times \GL_r$. For general $\bd$, it is more complicated to describe $Y_{\bd}$ directly.
\end{remark}

We define $Y_{\bd,n}$ as the set of isomorphism classes of ideals in $A_n$ given the topology induced by the surjection $\pi_n \colon X^{\circ}_{\bd,n} \to Y_{\bd,n}$. There are natural maps $Y_{\bd,n} \to Y_{\bd,n+1}$ and $Y_{\bd,n} \to Y_{\bd}$.

\begin{proposition}
The space $Y_{\bd}$ is the direct limit of the spaces $Y_{\bd,n}$.
\end{proposition}

\begin{proof}
It is clear that the set $Y_{\bd}$ is the direct limit of the sets $Y_{\bd,n}$. If $Z$ is a subset of $Y_{\bd}$, then the following are equivalent: 
\begin{itemize}
\item $Z$ is closed;
\item $\pi^{-1}(Z)$ is closed in $X_{\bd}^{\circ}$ (by definition of the topology on $Y_{\bd}$);
\item $\pi^{-1}(Z) \cap X_{\bd,n}^{\circ}$ is closed for all $n$ (by Corollary~\ref{cor:X*ind});
\item $\pi_n^{-1}(Z \cap Y_{\bd,n})$ is closed for all $n$ (it equals $\pi^{-1}(Z) \cap X_{\bd,n}^{\circ}$);
\item $Z \cap Y_{\bd,n}$ is closed for all $n$  (by definition of the topology on $Y_{\bd,n}$);
\item $Z$ is closed in the direct limit topology. 
\end{itemize}
Thus the topology on $Y_{\bd}$ is the direct limit topology.
\end{proof}

\begin{theorem*}[Theorem~\ref{thm:noeth}]
The space $Y_{\bd}$ is noetherian.
\end{theorem*}

\begin{proof}
Suppose $Z_{\bullet}$ is a descending chain of closed subsets in $Y_{\bd}$. Then $\pi^{-1}(Z_{\bullet})$ is a descending chain of $\GL$-stable Zariski closed subsets of $X^{\circ}_{\bd}$, and thus stabilizes by Corollary~\ref{cor:conjtop}. It follows that $Z_{\bullet}$ stabilizes, and so $Y_{\bd}$ is noetherian.
\end{proof}

\begin{remark}
An ideal invariant induces a function $Y_{\bd,n} \to \bZ \cup \{\infty\}$ for each $n$. An ideal invariant is cone-stable if and only if it is compatible with the transition maps $Y_{\bd,n} \to Y_{\bd,n+1}$. It follows that a cone-stable ideal invariant induces a function $Y_{\bd} \to \bZ \cup \{\infty\}$.
\end{remark}

\subsection{Finiteness of initial ideals}

For an ideal $I \subset A_n$ we let $\gin(I)$ denote the generic initial ideal of $I$ under the revlex order. We note that $\gin(\sigma(I))=\gin(I)$ for $\sigma \in \GL_n$, essentially by definition. The proof of the following theorem crucially depends on the theorem of Ananyan--Hochster \cite{ananyan-hochster} (Stillman's conjecture).

\begin{theorem} \label{thm:ginfin}
Given $\bd$ there exist $B$ and $C$ such that for any $n$ and any type $\bd$ ideal $I$ of $A_n$ the ideal $\gin(I)$ is generated by monomials of degree at most $C$ in the variables $x_1, \ldots, x_B$.
\end{theorem}

\begin{proof}
By Stillman's conjecture for $\bd$ \cite[Theorem C]{ananyan-hochster}, there exists a bound $B$ such that $\pdim(A_n/I) \le B$ for any type $\bd$ ideal $I$ of $A_n$. By a theorem of Caviglia~\cite[Theorem 2.4]{ms}, this implies that there exists $C$ such that $\reg(I) \le C$ for any such $I$ (the proof of that result shows that $C$ only depends on $B$ and not on the underlying field). Both $B$ and $C$ are independent of $n$.

Let $I$ be a type $\bd$ ideal of $A_n$ and put $s = \operatorname{depth}(A_n/I)$. After a general change of coordinates, we can assume that $x_{n}, x_{n-1}, \dots, x_{n-s+1}$ form a regular sequence on $A_n/I$. Since the projective dimension of $A_n/I$ is at most $B$, the Auslander--Buchsbaum Theorem~\cite[Theorem~19.9]{eisenbud} implies that $n-s\leq B$. By the Bayer--Stillman criterion~\cite[Theorem~2.4]{bayer-stillman}, the same $x_i$ form a regular sequence on the quotient of $A_n$ by the revlex initial ideal of $I$. Thus, $\gin(I)$ is definable in at most $B$ variables. Moreover~\cite[Corollary~2.5]{bayer-stillman} implies that the revlex generic initial ideal is definable in degree at most $C$.
\end{proof}

It is a consequence of \cite[Lemma~2.2]{bayer-stillman} that formation of revlex $\gin$ is cone-stable, and thus $\gin(I)$ is well-defined for any finitely generated homogeneous ideal $I \subset A$, and $\gin(I)$ is a finitely generated monomial ideal of $A$. We define a {\bf type $\bd$ revlex generic initial ideal} as an ideal of $A$ of the form $\gin(I)$ where $I$ is a type $\bd$ ideal. Theorem~\ref{thm:ginfin} yields:

\begin{corollary} \label{cor:ginfin}
There are only finitely many type $\bd$ revlex generic initial ideals.
\end{corollary}

\begin{remark}
Corollary~\ref{cor:ginfin} fails for other term orders.  See~\cite[Appendix A.2]{snellman}.
\end{remark}

\subsection{Hilbert numerators}

For a homogeneous ideal $I$ of $A_n$, the Hilbert series $\rH_{A_n/I}(t)$ can be written as a rational function $Q(t)/(1-t)^n$ with $Q(t)\in \bZ[t]$.  We call $Q(t)$ the {\bf Hilbert numerator} of $A_n/I$, and we denote it by $\HN_{A_n/I}(t)$. Terminology for $\HN_{A_n/I}(t)$ varies in the literature: our usage follows~\cite[p.~282]{kreuzer-robbiano}, but it is also called the $K$-polynomial~\cite[Definition 8.21]{miller-sturmfels}, and has other names elsewhere. Note that the Hilbert numerator is not necessarily the numerator of $\rH_{A_n/I}(t)$ when written in lowest terms.

The advantage of the Hilbert numerator is that it is cone-stable: $\HN_{A_{n+1}/IA_{n+1}}(t)=\HN_{A_n/I}(t)$.  We can thus define the Hilbert numerator of $A/I$ for any finitely generated homogeneous ideal $I$ of $A$.  We define a {\bf type $\bd$ Hilbert numerator} to be a polynomial of the form $\HN_{A/I}(t)$ for $I$ a type $\bd$ ideal of $A$.

\begin{theorem} \label{thm:hilbfin}
There are only finitely many type $\bd$ Hilbert numerators.
\end{theorem}

\begin{proof}
The Hilbert series associated to an ideal $I \subset A_n$ coincides with that of $\gin(I)$. It follows that the Hilbert numerator associated to an ideal $I \subset A$ coincides with that of $\gin(I)$, and so the result follows from Corollary~\ref{cor:ginfin}.
\end{proof}

\subsection{Two partial orders} \label{ss:order}

Given polynomials $f,g \in \bR[t]$, define $f<g$ if $f(x)<g(x)$ for all $0 < x < 1$. We use $f(x) \le g(x)$ to mean that either $f=g$ or that $f < g$, which is {\it not} the same as $f(x) \le g(x)$ for all $0 < x < 1$. Given series $f=\sum_{i \ge 0} a_i t^i$ and $g=\sum_{i \ge 0} b_i t^i$ in $\bR \lbb t \rbb$, define $f \preceq g$ if $a_i \le b_i$ for all $i$. Note that $f \le g \iff 0 \le g-f$, and similarly for $\preceq$.

\begin{proposition} \label{prop:order}
Let $f \in \bR[t]$. The following are equivalent:
\begin{enumerate}[\rm (a)]
\item $0 \le f$.
\item $f$ can be expressed in the form $\sum_{0 \le i,j \le N} c_{i,j} x^i (1-x)^j$ for some $N$ and non-negative coefficients $c_{i,j} \in \bR$.
\item There exists $N \ge 0$ such that $0 \preceq (1-t)^{-n} f(t)$ for all $n \ge N$.
\item There exists $n \ge 0$ such that $0 \preceq (1-t)^{-n} f(t)$.
\end{enumerate}
\end{proposition}

\begin{proof}
(a) $\Rightarrow$ (b). After a change of coordinates, this is \cite[Part 6, \S 6, Problem 49]{polyaszego}.

(b) $\Rightarrow$ (c). Write $f=\sum_{0 \le i,j \le N} c_{i,j} x^i (1-x)^j$ as in (b). Let $n \ge N$. Then $(1-t)^{-n} f(t) = \sum_{0 \le i,j \le N} c_{i,j} x^i (1-t)^{-(n-j)}$. Since $n-j \ge 0$ for all $j$ in the sum, the series $(1-t)^{-(n-j)}$ has non-negative coefficients. Since the $c_{i,j}$ are non-negative, it follows that $(1-t)^{-n} f(t)$ has non-negative coefficients.

(c) $\Rightarrow$ (d). Obvious.

(d) $\Rightarrow$ (a). If $f=0$ then obviously (a) holds. Thus suppose $f \ne 0$. Let $n$ be such that $0 \preceq (1-t)^{-n} f(t)$, and let $x \in (0,1)$. The series $(1-t)^{-n} f(t)$ is non-zero, has non-negative coefficients, and converges at $t=x$. Thus its value at $t=x$ is positive. Since $1-x$ is also positive, we conclude that $f(x)$ is positive, and so (a) holds.
\end{proof}

\subsection{The flattening stratification} \label{ss:strat}

Let $\{H_{\lambda}\}_{\lambda \in \Lambda}$ be the set of all type $\bd$ Hilbert numerators, where $\Lambda$ is a finite index set. Define a partial order on $\Lambda$ by $\mu < \lambda$ if $H_{\mu} < H_{\lambda}$, that is, if $H_{\mu}(x)<H_{\lambda}(x)$ for all $0<x<1$. For $\lambda \in \Lambda$, let $Y_{\bd}^{\lambda}$ be the locus in $Y_{\bd}$ where the corresponding ideal class has Hilbert numerator $H_{\lambda}$. Let $Y_{\bd,n}^{\lambda}=Y_{\bd}^{\lambda} \cap Y_{\bd,n}$.

\begin{proposition} \label{prop:strat}
For any $\lambda \in \Lambda$, the set $\bigcup_{\mu \ge \lambda} Y^{\mu}_{\bd}$ is closed in $Y_{\bd}$.
\end{proposition}

\begin{proof}
By Corollary~\ref{cor:X*ind} and the definition of the topology on $Y_{\bd,n}$, it suffices to show that $Z_n=\bigcup_{\mu \ge \lambda} \pi_n^{-1}(Y^{\mu}_{\bd,n})$ is closed in $X_{\bd,n}^{\circ}$ for all $n \gg 0$. Let $N$ be such that for any $\lambda,\mu \in \Lambda$ and $n \ge N$ we have $\lambda \le \mu$ if and only if $(1-t)^{-n} H_{\lambda}(t) \preceq (1-t)^{-n} H_{\mu}(t)$; this exists by Proposition~\ref{prop:order} and the fact that $\Lambda$ is finite. Let $n\geq N$ and let $\ff$ be a point of $X^{\circ}_{\bd,n}$. Then $\ff \in Z_n$ if and only if $H_{\lambda}(t) \le \HN_{B_{\ff}}(t)$, which in turn is equivalent to $(1-t)^{-n} H_{\lambda}(t) \preceq (1-t)^{-n} \HN_{B_{\ff}}(t)$. Since $(1-t)^{-n} \HN_{B_{\ff}}(t)$ is the Hilbert series $\rH_{B_{\ff,n}}(t)=\rH_{\cB_{\bd,n}(\ff)}(t)$, we see
\begin{displaymath}
Z_n=\{ \ff \in X_{\bd,n}^{\circ} \mid (1-t)^{-n} H_{\lambda}(t) \preceq \rH_{\cB_{\bd,n}(\ff)}(t) \},
\end{displaymath}
which is closed by the usual semi-continuity property of Hilbert series.
\end{proof}

\begin{corollary} \label{cor:locclosed}
Each $Y^{\lambda}_{\bd}$ is locally closed in $Y_{\bd}$.
\end{corollary}

Let $\ol{X}^{\lambda}_{\bd,n}$ be the closure of $\bigcup_{\mu \ge \lambda} \pi_n^{-1}(Y^{\mu}_{\bd,n})$ in $X_{\bd,n}$, endowed with the reduced subscheme structure.  Let $X^{\lambda}_{\bd,n}$ be the complement of $\bigcup_{\mu>\lambda} \ol{X}^{\mu}_{\bd,n}$ in $\ol{X}^{\lambda}_{\bd,n}$, considered as an open subscheme of $\ol{X}^{\lambda}_{\bd,n}$. The set $X^{\lambda,\circ}_{\bd,n}$ of closed points in $X^{\lambda}_{\bd,n}$ is exactly $\pi_n^{-1}(Y^{\lambda}_{\bd,n})$.

For simplicity, we have only defined $Y_{\bd}$ as a topological space (as opposed to an ind-stack). To make sense of the flatness statement from Theorem~\ref{thm:strat}, we thus pass to the cover $X_{\bd}$, and interpret Theorem~\ref{thm:strat} as the following concrete statement:

\begin{proposition} \label{prop:flatstrat}
The restriction of $\cB_{\bd,n}$ to $X=X^{\lambda}_{\bd,n}$ is $\cO_X$-flat.
\end{proposition}

\begin{proof}
For $\ff \in X$ we have $\rH_{\cB_{\bd,n}(\ff)}(t)=\rH_{B_{\ff,n}}(t)=(1-t)^{-n} H_{\lambda}(t)$. Thus the Hilbert series of the fiber of $\cB_{\bd,n}$ is constant on $X$. Let $\cF$ be one of the graded pieces of $\cB_{\bd,n}$. Then $\cF$ is a coherent sheaf on $X$ whose fiber at all closed points has the same dimension, say dimension $d$. By semi-continuity of fiber dimension, the locus of (not necessarily closed) points where the fiber dimension is $\ne d$ is the union of a closed set (where the dimension is $>d$) and an open set (where the dimension is $<d$).  Yet this locus contains no closed points, and since $X$ is of finite type over a field, this implies that the fiber of $\cF$ has the same dimension on the non-closed points as well, which in turn implies that $\cF$ is locally free \cite[Ex.~20.14(b)]{eisenbud}.
\end{proof}

\begin{corollary} \label{cor:flatstrat}
Let $\nu$ be a cone-stable weakly upper semi-continuous ideal invariant, and let $n \in \bZ$ be given. Then $Z=\{ I \in Y^{\lambda}_{\bd} \mid \nu(I) \ge n \}$ is a closed subset of $Y^{\lambda}_{\bd}$.
\end{corollary}

\begin{proof}
It suffices to show that $Z'=\pi_n^{-1}(Z \cap Y_{\bd,n})$ is a closed subset of $X_{\bd,n}^{\lambda,\circ}$ for all $n$. We have $Z'=\{ \ff \in X_{\bd,n}^{\lambda,\circ} \mid \nu(I_{\ff,n}) \ge n\}$. This is closed, since $\cB_{\bd,n}$ is flat over $X_{\bd,n}^{\lambda}$ and $\nu$ is weakly upper semi-continuous.
\end{proof}

\subsection{Proof of Theorem~\ref{thm:invariants}} \label{sec:cone stable}

Fix a cone-stable weakly upper semi-continuous ideal invariant $\nu$, and let $\bd$ be given. Let $Z_k \subset Y_{\bd}$ be the locus defined by $\nu \ge k$. Observe that $Z_k \cap Y_{\bd}^{\lambda}$ is closed in $Y_{\bd}^{\lambda}$ by Corollary~\ref{cor:flatstrat}, and that the space $Y_{\bd}^{\lambda}$ is noetherian, being a subspace of the noetherian space $Y_{\bd}$ (see Theorem~\ref{thm:noeth}). We thus see that the descending chain $Z_{\bullet} \cap Y_{\bd}^{\lambda}$ stabilizes. Since there are only finitely many $\lambda$, it follows that the chain $Z_{\bullet}$ stabilizes. Let $N$ be such that $Z_k=Z_N$ for all $k \ge N$. We thus see that $\nu \ge N$ implies $\nu \ge k$ for all $k \ge N$; thus $\nu \ge N$ implies $\nu=\infty$. We therefore find that $\nu<N$ or $\nu=\infty$ holds at all points in $Y_{\bd}$, and so $\nu$ is bounded in degree $\bd$. \hfill \qedsymbol

\begin{remark}\label{remark:strat compare}
The stratification of $X_{\bd}$ differs from the stratification produced in \cite[Theorem~5.13]{ess-stillman} in two ways.  First, each $\bk$-point of $X_{\bd}$ corresponds to a tuple $f_1,\dots,f_r$ in a polynomial ring $\bk[x_1,x_2,\dots,x_n]$ for some $n$, whereas \cite{ess-stillman} allows ``infinite polynomials'', i.e., tuples which lie in the graded inverse limit ring.  Second, since the Betti table determines the Hilbert numerator, the stratification in \cite{ess-stillman} refines the above stratification of $X_{\bd}$.
\end{remark}

\section{Examples of ideal invariants} \label{s:examples}

\subsection{The number of linear subspaces in a variety}\label{ss:fano}
In this section we work over an algebraically closed field $\bk$.  A smooth cubic surface in $\bP^3$ contains exactly $27$ lines.  An arbitrary cubic surface $Y\subset \bP^3$ can contain fewer than $27$ lines or it can contain an infinite number of lines (e.g., if $Y$ is reducible); but if $Y$ contains a {\em finite} number of lines, then it contains at most $27$ lines~\cite[Theorem~9.48]{milne}. In this section, we prove a sort of generalization of this.

Fix a non-negative integer $c$. Let $\Gr_c(\bk^n)$ be the Grassmannian of codimension $c$ subspaces of $\bk^n$. For a homogeneous ideal $I \subset A_n$, let $\fU_I$ be the closed subscheme of $\Gr_c(\bk^n)$ whose $T$-points are those families of subspaces of $\bk^n$ scheme-theoretically contained in $T \times \Spec(A_n/I)$. Thus $\fU_I(\bk)$ is exactly the set of subspaces of $\bk^n$ of codimension $c$ contained in $V(I)$. We say that a point $x$ of $\fU_I(\bk)$ is {\bf rigid} if the Zariski tangent space to $\fU_I$ at $x$ vanishes. Define an ideal invariant $\nu$ as follows: $\nu(I)=\infty$ if $\fU_I$ contains a non-rigid point; otherwise, $\nu(I)=\# \fU_I(\bk)$. We note that $\nu(I)<\infty$ if and only if $\fU_I$ is \'etale over $\bk$.

\begin{proposition}\label{prop:fano}
The ideal invariant $\nu$ is degreewise bounded.
\end{proposition}
\addtocounter{equation}{-1}
\begin{subequations}
\begin{proof}
We first show that $\nu$ is ``eventually'' cone-stable. Let $I \subseteq A_n$ with $n>c$ and let $I'=IA_{n+1}$. We show $\nu(I)=\nu(I')$. We have $V(I')=V(I) \times \bA^1$, and so we have a map
\begin{equation} \label{eq:fano}
\fU_I \to \fU_{I'}, \qquad U \mapsto U \times \bA^1.
\end{equation}
This map is clearly a closed immersion. We claim that if $\fU_I$ is finite over $\bk$ then this is an isomorphism on $\bk$-points. Let $U' \in \fU_{I'}(\bk)$, and let $U$ be its projection to $\bA^n$. Then $U$ is a linear subspace of $V(I) \subset \bA^n$ containing the origin. If $U$ has codimension $c$ then $U'=U \times \bA^1$, and so $U'$ is in the image of \eqref{eq:fano}. If not, $U$ has codimension strictly less than $c$, and thus dimension at least~2 (since $c<n$), and it must contain an infinite number of linear subspaces of codimension $c$ in $\bA^n$, which contradicts the finiteness of $\fU_I$.

If $\nu(I)=\infty$ then $\fU_I$ is not \'etale, and so $\fU_{I'}$ is not \'etale, and so $\nu(I')=\infty$. Suppose now that $\nu(I)$ is finite. Then \eqref{eq:fano} is an isomorphism on $\bk$-points. A $\bk$-point $x$ of $\fU_I$ corresponds to a surjection $B=A_n/I \to \bk[t_1, \ldots, t_{n-c}]=C$ of graded rings. The Zariski cotangent space of $x \in \fU_I$ is identified with $\Hom_B(J/J^2, C)$, where $J$ is the kernel of $B \to C$, and the maps are taken as graded $B$-modules. Since $\nu(I)$ is finite, this $\Hom$ space therefore vanishes. It follows that $\Hom_{B[t_{n+1}]}(J/J^2[t_{n+1}], C[t_{n+1}])$ also vanishes, which is the cotangent space of the image of $x$ under \eqref{eq:fano}. We thus see that $\fU_{I'}$ is \'etale, and so \eqref{eq:fano} is an isomorphism of schemes. Thus $\nu(I)=\nu(I')$.

We now prove weak semi-continuity. Suppose that $\cI \subset \cO_S \otimes A_n$ is a family of ideals over $S$, flat over $S$, and write $\cI_s$ for the ideal of $A_n$ at $s \in S(\bk)$. The construction of $\fU$ works in families: we have a scheme $\fU_{\cI}$ over $S$ whose fiber at $s \in S(\bk)$ is $\fU_{\cI_s}$. Since $\fU_{\cI}$ is a closed subscheme of the Grassmannian, it is proper over $S$.

By Lemma~\ref{lem:fano} below, the set of points $s \in S(\bk)$ at which $(\fU_{\cI})_s$ is \'etale is open. We thus see that the locus $\nu \ge \infty$ is closed.

We now show $\nu \ge k$ is closed, for $k<\infty$. Since the locus $\nu \ge \infty$ is closed, we can discard it from $S$. Thus we can assume $(\fU_{\cI})_s$ is \'etale over $\bk$ for all $s \in S(\bk)$. Since $\fU_{\cI}$ is also quasi-finite over $S$, it is finite over $S$. It follows that the fibral dimension of $\cO_{\fU_{\cI}}$, which is just $\nu$, is upper semi-continuous.

The proof of Theorem~\ref{thm:invariants} shows that $I \mapsto \nu(I)$ is degreewise bounded for ideals $I$ in $A_n$ with $n>c$. For $n \le c$ we have $\nu(I) \le 1$. Thus $\nu$ is degreewise bounded.
\end{proof}
\end{subequations}

\begin{lemma} \label{lem:fano}
Let $X \to S$ be a flat, proper map of schemes, with $S/\bk$ of finite type. Then the locus of points $s \in S(\bk)$ where $X_s$ is \'etale over $\bk$ is open.
\end{lemma}

\begin{proof}
Since $\bk$ is algebraically, the locus where $X_s$ is \'etale is the intersection of the locus where $X_s$ is dimension $<1$ and the locus where $X_s$ is reduced.  These loci are open by~\cite[0D4I and 0C0E]{stacks}.
\end{proof}

\begin{remark}
We believe the ideal invariant $I \mapsto \# \fU_I(\bk)$ is degreewise bounded. It is cone-stable, but not upper semi-continuous. 
By contrast, we could also consider the scheme structure on $\fU_I$, setting $\tau(I)=\deg \fU_I$ when $\fU_I$ is finite and $\tau(I)=\infty$ else.  However, the invariant $\tau$ is upper semi-continuous, but not cone-stable.
\end{remark}

\subsection{Invariants of singularities}

Our ideal invariants are integer-valued. However, there are interesting ideal invariants that are not integer-valued. The methods used in the proof of Theorem~\ref{thm:invariants} can sometimes be applied to these invariants. Here is an example:

\begin{proposition}
Suppose $\bk$ is a field of characteristic~$0$, and fix $\bd$. Let $\Lambda$ be the set of rational numbers occurring as the log-canonical threshold $($at the cone point$)$ of some type $\bd$ ideal. Then $\Lambda$ satisfies the descending chain condition. 
\end{proposition}

\begin{proof}
Cone-stability follows from the definition; for instance, using the analytic definition (as in~\cite[Theorem~1.2]{mustata-impanga}), a rational function $f$ is integrable around the origin of $\bA^{n}$ if and only if the corresponding function is integrable around the origin of $\bA^{n+1}$.  Lower semi-continuity follows from \cite[Example~9.5.41]{lazarsfeld}. For each $\lambda \in \Lambda$ and each $\ff\in X_{\bd, n}$ we say that $\ff$ lies in $Z_{\lambda}$ if and only if the log canonical threshold of  $(f_1,\dots,f_r)$ at the cone point in $\bA^n$ is at most $\lambda$.  Since $Z$ is $\GL$-invariant, and since $Z_{\lambda}\cap X_{\bd,n}$ is closed by semicontinuity, it follows that $Z_{\lambda}$ is closed.  An infinite decreasing chain of values in $\Lambda$ would thus yield an infinite decreasing chain of $\GL$-stable closed subsets in $X_{\bd}$, contradicting Theorem~\ref{thm:noeth}.
\end{proof}

\begin{remark}
A similar statement holds for $F$-pure thresholds in positive characteristic. Here the semi-continuity follows from \cite[Theorem~5.1]{mustata-yoshida}.
\end{remark}

\subsection{Previously known degreewise bounded invariants}

Many ideal invariants have been previously shown to be degreewise bounded. We catalogue some here to hint at the ubiquity of the phenomenon.

\begin{proposition}
The following invariants of an ideal $I$ are known to be degreewise bounded by previous results in the literature:
\begin{enumerate}[\rm (1)]
    \item\label{list:bezout}  The degree of $I$.
	\item\label{list:max codim}  The maximal codimension of a minimal or associated prime of $I$.
	\item\label{list:pdim}  The projective dimension of $I$.
	\item \label{list:reg}  The Castelnuovo--Mumford regularity of $I$.
	\item\label{list:betti}  The Betti number $\beta_{i,j}(I)$ for any $i,j$.
	\item\label{list:minimal sums} The sum of the degrees of: the minimal primes of $I$ or the minimal primary components of $I$. 
	\item  The $r$th arithmetic degree of $I$, as defined in \cite[Definition 3.4]{bayer-mumford}.\footnote{Since the embedded primary components of an ideal are not uniquely defined, it would not make sense to talk about the degree of an embedded primary component.  To remedy this, \cite[\S3]{bayer-mumford} introduces a notion of the multiplicity of an embedded component of $I$ that depends only on the ideal $I$ and the corresponding associated prime.  This leads to their definition of the $r$th arithmetic degree of an ideal.}
	\item\label{list:num primes}  The number of: minimal primes, embedded primes, or associated primes of $I$.
	\item\label{list:radical} The degree of $\sqrt{I}$.
	\item\label{list:symbolic}  The minimal $B$ such that the symbolic power $I^{(Br)}$ belongs to $I^r$ for all $r$.
	\item\label{list:null}  The minimal $B$ such that $\sqrt{I}^B\subseteq I$.
	\item\label{list:numgens}  The largest degree of a generator, or the number of generators, for: any associated prime of $I$, the radical $\sqrt{I}$, or the symbolic power $I^{(r)}$ for any integer $r$.
\end{enumerate}
\end{proposition}

\begin{proof}
\begin{enumerate}[label=(\arabic*),leftmargin=*]
\item Refined Bezout's Theorem~\cite[Theorem~12.3]{fulton}.
\end{enumerate}
\begin{enumerate}[label=(\arabic*),resume,topsep=0pt]
\item The principal ideal theorem~\cite[Theorem~10.2]{eisenbud}.  
\item Stillman's conjecture~\cite[Theorem~C]{ananyan-hochster}.
\item Stillman's conjecture combined with Caviglia's theorem~\cite[Theorem~2.4]{ms}.  
\item We can combine (3) and (4) with Boij--S\"oderberg theory \cite[Theorems 0.1, 0.2]{eisenbud-schreyer} to see that only finitely many Betti tables for $\beta(S/I)$ are possible, and the statement follows.
\item For either sum, one can apply Refined Bezout's Theorem~\cite[Theorem~12.3]{fulton}.
\item  We use \cite[Proposition~3.6]{bayer-mumford}.
\item The number of minimal primes is bounded by (6), so it suffices to bound the number of embedded primes.  By (2), we know that the codimension of an embedded prime of $I$ can take on only finitely many distinct values, so it suffices to bound the number of embedded primes of a given codimension. Since each embedded prime of codimension $n-r$ contributes at least $1$ to the $r$th arithmetic degree of $I$, we can then apply (7).
\item Follows from (6).
\item Follows from (2) plus  
\cite[Theorem 1.1(c)]{hochster-huneke}.
\item Follows from the Effective Nullstellensatz~\cites{brownawell,kollar,sombra}.  
\item For an associated prime $P$ of $I$, or for a minimal primary component $Q$ of $I$, this follows from \cite[Theorem~D(b)]{ananyan-hochster}.  Since $\sqrt{I}$ is the intersection of the minimal primes of $I$, it suffices to show that we can bound the number and degree of defining equations of $J\cap J'$ in terms of the number and degree of defining equations of $J$ and $J'$.  Using parts (3) and (4) above, we can bound the regularity and projective dimension of $J, J'$ and $J+J'$.  Using the exact sequence relating these ideals to $J\cap J'$, we can bound the regularity and projective dimension of $J\cap J'$ as well.  There are thus only a finite number of possible revlex gins of $J\cap J'$, yielding the desired bounds.  This proves the statement for $\sqrt{I}$; a similar argument works for symbolic powers. \qedhere
\end{enumerate}
\end{proof}
\begin{remark}
Combining parts \ref{list:numgens} and \ref{list:reg} answers ~\cite[Problem 3.9]{peeva-stillman}.  See also \cites{cd,ravi} for related work on the degreewise boundedness of taking radicals.
\end{remark}

\section{Additional comments} \label{s:comments}

\subsection{A converse theorem}

Before Draisma's paper \cite{draisma} appeared, we had Theorem~\ref{thm:invariants} in the form ``if sums of symmetric power functors are topologically noetherian then ideal invariants are bounded.'' 
The converse of this statement, that boundedness of ideal invariants implies the noetherianity of of the space $Y_{\bd}$, follows from an equivariant version of the Hilbert basis theorem that appears in \cite[\S1.2]{eqhilb}.
 
\subsection{An improvement to Theorem~\ref{thm:noeth}} \label{ss:improve}

Let $Y_{\le d}$ be the set of all isomorphism classes of finitely generated homogeneous ideals of $A=\bk[x_1,x_2,\ldots]$ that are generated in degrees at most $d$ (but with no condition on the number of generators). We have a surjective map 
\[
X_{\le d} = \bigoplus_{i=1}^d \left( \Sym^i(\bk^\infty) \otimes \bk^\infty \right) \to Y_{\le d}
\]
defined by sending $\sum_{i=1}^d \sum_j f_{i,j} \otimes r_{i,j}$ to the ideal generated by the $f_{i,j}$. Alternatively, we may view an element in $X_{\le d}$ as a finite rank map $\bk^\infty \to \bigoplus_{i=1}^d \Sym^i(\bk^\infty)$ and the ideal is generated by the image. We can topologize $Y_{\le d}$ as a quotient space of $X_{\le d}$. The following statement greatly strengthens Theorem~\ref{thm:noeth}:

\begin{theorem*}[Theorem~\ref{thm:noeth 2}]
The space $Y_{\le d}$ is noetherian.
\end{theorem*}

\begin{proof}
The space $X_{\le d}$ carries a natural action of $\GL \times \GL$, and the map $X_{\le d} \to Y_{\le d}$ is equivariant if we have this group act trivially on $Y_{\le d}$. Consider the diagonal action of $\GL$. The tensor product $\Sym^i(\bk^\infty) \otimes \bk^\infty$ is a finite length representation for each $i$. So by \cite{draisma}, $X_{\le d}$ is topologically $\GL$-noetherian.  
  Since the map $X_{\le d} \to Y_{\le d}$ is a quotient map, $Y_{\le d}$ is also noetherian (since $\GL$ acts trivially on it).
\end{proof}

We give two consequences of the theorem.

\begin{proposition}
Let $\cH$ be the set of all polynomials of the form $\HN_{A/I}(t)$ with $I$ a homogeneous ideal finitely generated in degrees at most $d$. Endow $\cH$ with the partial order $\le$ from \S \ref{ss:order}. Then $\cH$ satisfies the ascending chain condition.
\end{proposition}

We say that an ideal invariant is {\bf strongly upper semi-continuous} if for any family of ideals $\cI$ over $S$ (with no flatness condition imposed), the function $s \mapsto \nu(\cI_s)$ is upper semi-continuous. (Here $\cI_s$ denotes the ideal at $s$, which is a homomorphic image of the fiber.) We say that $\nu$ is {\bf strongly degreewise bounded} if for every $d$ there exists a $B$ such that $\nu(I) \le B$ for any homogeneous ideal $I$ finitely generated in degrees at most $d$.  By a straightforward adaptation of the proof of Theorem~\ref{thm:invariants}, we also obtain:

\begin{proposition} \label{prop:invar2}
Any cone-stable strongly upper semi-continuous ideal invariant is strongly degreewise bounded.
\end{proposition}

\begin{remark}
The analogue of Theorem~\ref{thm:strat} fails for $Y_{\le d}$: even $Y_{\le 1}$ would need a separate stratum for each integer $c$ consisting of the isomorphism class of the ideal $\langle x_1, x_2,\dots,x_c\rangle$.
\end{remark}
\begin{remark}
Let $Y$ be the set of isomorphism classes of all finitely generated ideals in $A$. We claim that $Y$ is not noetherian. Indeed, if it were then the set of all polynomials of the form $\HN_{A/I}$, with $I$ any finitely generated homogeneous ideal, would satisfy the ascending chain condition. But it does not: indeed, $1-t^d$ is a Hilbert numerator for any $d \ge 0$, and these form an ascending chain.
\end{remark}

\subsection{Boundedness of tca ideal invariants}

Draisma's theorem states that if $V$ is any finite length polynomial representation of $\GL$ then $\Spec(\Sym(V))$ is topologically noetherian. In our proof of Theorem~\ref{thm:noeth}, we only applied this result with $V$ being a finite sum of symmetric powers. It is natural to wonder, therefore, if the remaining cases of Draisma's theorem have implications for ideal invariants. We now give one possible answer to this question.

Recall that a {\bf twisted commutative algebra} (tca) over $\bk$ is a commutative associative unital graded $\bk$-algebra $A=\bigoplus_{k \ge 0} A_k$ equipped with an action of the symmetric group $S_k$ on $A_k$ satisfying certain conditions, the most important being that multiplication is commutative up to a ``twist'' by $S_k$; see \cite{expos} for the full definition. Let $A_n=\bk\langle x_1, \ldots, x_n \rangle$ be the tca freely generated by $n$ indeterminates of degree one. Explicitly, $A_n$ is the tensor algebra on $\bk^n$, equipped with the natural action of $S_k$ on the $k$th tensor power.

We define a {\bf tca ideal invariant} to be a rule associating to every ideal $I$ in any $A_n$ a quantity $\nu(I) \in \bZ \cup \{\infty\}$, depending only on $(A_n, I)$ up to isomorphism. One can then prove:

\begin{theorem} \label{thm:tca}
Any cone-stable strongly upper semi-continuous tca ideal invariant is degreewise bounded.
\end{theorem}

The proof is similar to before. We define a space $Y_{\bd}$ parametrizing ideals of type $\bd$ in $A=\bk\langle x_1, x_2, \ldots \rangle$. Using Draisma's theorem (now applied to arbitrary polynomial functors), we deduce that $Y_{\bd}$ is noetherian. We no longer have the analog of Theorem~\ref{thm:strat} (this is an interesting open problem), which is why we restrict to strongly upper semi-continuous invariants in the theorem: this ensures that the locus $Z_n \subset Y_{\bd}$ where $\nu \ge n$ is closed.

The projective dimension of any non-trivial ideal in the tca $A_n$ is infinite. However, it is known \cite{symu1,ganli} that the regularity of any ideal is finite. One can therefore formulate the following tca analog of Stillman's conjecture:

\begin{question}[TCA Stillman]
Is the tca ideal invariant ``regularity'' degreewise bounded?
\end{question}

Regularity is only weakly semi-continuous, so Theorem~\ref{thm:tca} does not apply to this question. We do not know if a positive answer to this question would imply a version of Theorem~\ref{thm:strat} in this setting, as the tools used in the proof of Theorem~\ref{thm:ginfin} do not yet exist for tca's.

\subsection{Degreewise bounded invariants of modules}\label{subsec:fin modules}

Fix a doubly indexed sequence $\bd=(d_{1,1},d_{1,2},\dots,d_{1,b},d_{2,1},\dots,d_{a,b})$ in $\bZ^{a \times b}$.  We say that a graded module is {\bf type $\bd$} if it admits a presentation matrix
\[
\begin{pmatrix}
f_{1,1}&f_{1,2}&\cdots&f_{1,b}\\
f_{2,1}&f_{2,2}&\cdots&f_{2,b}\\
\vdots & &\ddots& \vdots\\
f_{a,1}&f_{a,2}&\cdots  &f_{a,b}
\end{pmatrix}
\]
where $\deg(f_{i,j})=d_{i,j}$ for all $i,j$. A {\bf module invariant} is a rule that associates to every graded module $M$ over any standard-graded polynomial ring $A$ a quantity in $\bZ\cup \{\infty\}$ that only depends on the pair $(A, M)$ up to graded isomorphism. We extend the notions of cone-stability, weak upper semi-continuity, and degreewise boundedness in the natural way.

\begin{theorem}\label{thm:inv modules}
Any module invariant that is cone-stable and weakly upper semi-continuous is degreewise bounded.
\end{theorem}

The proof is similar to that of Theorem~\ref{thm:invariants}, so we omit the details. Similarly, analogues of the results in \S\ref{ss:improve} can be obtained.

\end{document}